\documentclass{article}

\usepackage[english]{babel}
\usepackage{amsfonts,dsfont,mathrsfs}
\usepackage{amsmath,amssymb,amsthm,amsfonts}
\usepackage{graphicx}
\usepackage{caption}
\usepackage{xcolor} 
\usepackage{makeidx}

\usepackage{epsfig}
\usepackage{graphicx}
\usepackage{graphics}
\usepackage{float}
\usepackage{multirow}
\usepackage{color}
\usepackage{fullpage}
\usepackage[normalem]{ulem} 
\usepackage{makeidx}
\usepackage{xspace}
\usepackage{cite}
\usepackage{enumitem,hyperref,graphicx}
\usepackage{nameref}

\makeindex

\usepackage{lineno}

\def\namedlabel#1#2{\begingroup
    #2%
    \def\@currentlabel{#2}%
    \phantomsection\label{#1}\endgroup}

\makeatletter

\newtheorem{Theorem}{Theorem}[section]
\newtheorem{Lemma}{Lemma}[section]

\newtheorem{lemma}{Lemma}
\newtheorem{definition}{Definition}
\newtheorem{proposition}{Proposition}[section]

\definecolor{darkred}{rgb}{1, 0.1, 0.3}
\definecolor{darkblue}{rgb}{0.1, 0.1, 1}
\definecolor{darkgreen}{rgb}{0,0.6,0.5}

\renewcommand{\H}{\mathcal{H}}

\makeatletter
\def\namedlabel#1#2{\begingroup
    #2%
    \def\@currentlabel{#2}%
    \phantomsection\label{#1}\endgroup}

\usepackage[colorinlistoftodos]{todonotes}
\begin{document}

\title{Well-posedness for integro-differential sweeping processes of Volterra type}
 
\author{Emilio Vilches\footnote{Instituto de Ciencias de la Ingenier\'ia, Universidad de O'Higgins and Center for Mathematical Modeling, Universidad de Chile, Chile. Email: emilio.vilches@uoh.cl.} }

%\institute{Stockholm University, Stockholm, Sweden\\\email{elena.touli@math.su.se}\andthe Ohio State University Columbus, Ohio, U.S.A.\\\email{yusu@cse.ohio-state.edu}\\}\\

%\begin{document}
\maketitle
%\linenumbers
%\setcounter{page}{0}

\begin{abstract}
In this paper, we study the well-posedness of integro-differential sweeping processes of Volterra type. Using new  enhanced versions of Gronwall's inequality, a reparametrization technique, and a fixed point argument for history-dependent operators, we obtain the existence of solutions and provide a fully continuous dependence result for the integro-differential sweeping process.  The paper ends with an application to projected dynamical systems.
\end{abstract}

\section{Introduction}%\label{sec0}
Given a real Hilbert space $\H$, the Moreau's sweeping process is a differential inclusion involving the normal cone to a family of closed moving sets. Since its introduction by J.J. Moreau in \cite{MO1,MO2}, the sweeping process has been an important mathematical framework for the modeling for several problems in contact mechanics, electrical circuits, and crowd motion, among others (see, e.g., \cite{Acary-Bon-Bro-2011,Brogliato-M,Maury-Venel}). Moreover, so far, there is a well-developed existence theory for extensive classes of nonconvex sets.

In this paper, we are interested in a new variant of the sweeping process called the integro-differential sweeping process of Volterra type. In its general form, it is the following differential inclusion:
\begin{equation}\label{SP1}
\left\{
\begin{aligned}
\dot{x}(t)& \in -N_{C(t)}(x(t))+f(t,x(t))+\int_{T_0}^t g(t,s,x(s))ds &  \textrm{ a.e. } t\in I,\\
x(T_0)&=x_0\in C(T_0),
\end{aligned}
\right.
\end{equation}
where $C(t)$ is a nonempty, closed and $\rho$-uniformly prox-regular set for all $t\in I$ and $f\colon I\times \H \to \H$ and $g\colon I\times I\to \H$ are functions satisfying assumptions \ref{Hf} and \ref{Hg} below (see Section \ref{sec2}).  The integro-differential sweeping process \eqref{SP1} is a natural extension of the sweeping process. It was introduced by Brenier, Gangbo, Savar\'e and Westdickenberg \cite{MR3039208} to describe a one-dimensional compressible  fluid under the influence of an endogenous force field. Then, Colombo and Kozaily \cite{MR4099068} obtain the existence of solutions for a particular version of \eqref{SP1} by applying the Moreau-Yosida regularization. A great step was given by Bouach, Haddad and Thibault \cite{MR4492538,MR4421896}, who develop a complete existence theory  for the  integro-differential sweeping process using an appropriate catching-up algorithm.

In this paper, we show that the well-posedness for the integro-differential sweeping process can be obtained through a reparametrization technique and a fixed point argument for history-dependent operators. 
Moreover, we provide a fully continuous dependence result with respect to the data of the problem.  Our approach is based in some enhanced version of the Gronwall's inequality. 

The paper is organized as follows. After some mathematical preliminaries, in Section \ref{sec1}, we prove two new versions of Gronwall's inequality, which are used, in Section \ref{sec2},  to obtain the existence of solutions for the integro-differential sweeping process. Our approach is based on a fixed point argument. In Section \ref{sec3}, we provide a fully continuous dependence result for the integro-differential sweeping process.  The paper ends with an application to projected dynamical systems.

\section{Mathematical Preliminaries}\label{sec0}
Let $\mathcal{H}$ be a real Hilbert space endowed with an inner product $\langle\cdot,\cdot\rangle$ and associated norm $\Vert \cdot \Vert$. As usual, $\mathbb{R}$ will denote the set of real numbers and  $\mathbb{R}+$ the set of nonnegative reals, that is, $\mathbb{R}_+:=[0,+\infty[$. It will be convenient to use the notation $I_=[T_0,T]$ for $-\infty<T_0<T<+\infty$.

Given a set $S\subset \H$ and $x\in S$. We say that $v$ belongs to the \emph{proximal normal cone} $N_S^P(x)$ if there exists $\sigma\geq 0$ such that
$$
\langle v,y-x\rangle \leq \sigma \Vert y-x\Vert^2 \textrm{ for all } y\in S.
$$
If $x\notin C$, we set $N_S^P(x)=\emptyset$. We refer to \cite{Clarke1998,Penot-2012} for more details and properties.

Now, we recall the concept of uniformly prox-regular sets. Introduced by Federer in the finite-dimensional setting (see \cite{MR110078}) and later developed by Rockafellar, Poliquin, and Thibault in \cite{MR1694378},   the prox-regularity generalizes and unifies convex sets and nonconvex bodies with $C^2$ boundary. We refer to \cite{MR2768810,Thibault-2023-II} for a survey.
\begin{definition}
    Let $S$ be a closed subset of $\H$ and $\rho\in ]0,+\infty]$. The set $S$ is called $\rho$-uniformly prox-regular if for all $x\in S$ and $\zeta\in N^P(S;x)$ one has
    \begin{equation*}
        \langle \zeta,x'-x\rangle\leq \frac{\|\zeta\|}{2\rho}\|x'-x\|^2 \textrm{ for all } x'\in S.
    \end{equation*}
\end{definition}
\noindent It is important to emphasize that convex sets are $\rho$-uniformly prox-regular for any $\rho>0$. 

The following result provides a characterization of uniformly prox-regular sets (see, e.g., \cite[Theorem~16]{MR2768810}).
\begin{proposition}\label{prox-regularity} Let $S$ be a closed subset of $\H$ and $\rho\in ]0,+\infty]$. Then, the following assertions are equivalent.
\begin{itemize}
\item[(a)] The set $S$ is $\rho$-uniformly prox-regular.
\item[(b)] For any $x_i\in S$, $v_i\in N^P(S;x_i)\cap \mathbb{B}$ with $i=1,2$ one has
$$
\langle v_1-v_2,x_1-x_2\rangle \geq -\frac{1}{\rho}\Vert x_1-x_2\Vert^2,
$$
that is, the set-valued mapping $N^P(S;\cdot)\cap\mathbb{B}$ is $1/\rho$-hypomonotone.
\end{itemize}
\end{proposition}
It is well-known that the proximal normal cone coincides with Fr\'echet, Limiting, and Clarke normal cones (see \cite{MR2768810}). Hence, we simply denote the proximal normal cone to a prox-regular set $S$ as $N(S;\cdot)=N_S(\cdot)$. 

\begin{lemma}\label{lemma-derivative}
Let $w_1, w_2\colon I\to \H$ be two absolutely continuous functions such that
$$
\frac{1}{2}\frac{d}{dt}\Vert w_1(t)-w_2(t)\Vert^2 \leq \alpha(t)\Vert w_1(t)-w_2(t)\Vert^2+\beta(t)\Vert w_1(t)-w_2(t)\Vert \textrm{ for a.e. } t\in I,
$$
where $\alpha, \beta\colon I\to \mathbb R$ are nonnegative functions. Then, for a.e. $t\in I$, it holds
\begin{equation}\label{normad}
\frac{d}{dt}\Vert w_1(t)-w_2(t)\Vert \leq \alpha(t)\Vert w_1(t)-w_2(t)\Vert+\beta(t).
\end{equation}
\end{lemma}
\begin{proof} Let $h(t):=\Vert w_1(t)-w_2(t)\Vert$ and let $\Omega$ be the set of full measure in which  $\dot{h}$ exists.  Let us consider the sets
$$
\Omega_1:=\{ t\in \Omega \colon w_1(t)\neq w_2(t)\} \textrm{ and } \Omega_2:=\{ t\in \Omega\colon w_1(t)= w_2(t)\}.$$
On the one hand, for  $t\in \Omega_1$, one has
$$
\frac{1}{2}\frac{d}{dt}\Vert w_1(t)-w_2(t)\Vert^2=\Vert w_1(t)-w_2(t)\Vert \frac{d}{dt}\Vert w_1(t)-w_2(t)\Vert,
$$
which implies that \eqref{normad} holds for $t\in \Omega_1$. On the other hand, for any $t\in \Omega_2$, the map $t\mapsto h(t)$ attains a minimum at this point. Thus, for $t\in \Omega_2$, we have
$$\frac{d}{dt}\Vert x_1(t)-x_2(t)\Vert =0,$$
which implies that \eqref{normad} holds for $t\in \Omega_2$. The proof is then complete.
\end{proof}

\noindent  Given a closed interval $J\subset \mathbb{R}$, we denote by $C(J;\H)$ the space of continuous mappings from $J$ to $\H$ equipped with the supremum norm $\Vert \cdot\Vert_{\infty}$. We end this section with a fixed point theorem for history-dependent operators. We refer to \cite[p.~41-45]{MR3752610} for a proof.
\begin{proposition}\label{History}
Let $\mathcal{F}\colon C([0,T];\H)\to C([0,T];\H)$ be a history-dependent operator, i.e., there exists $L_{\mathcal{F}} \geq 0$ such that
$$
\Vert \mathcal{F}(y_1)(t)-\mathcal{F}(y_2)(t)\Vert \leq L_{\mathcal{F}}\int_{0}^t \Vert y_1(s)-y_2(s)\Vert ds \textrm{ for all } y_1, y_2 \in C([0,T];\H) \textrm{ and } t\in [0,T].
$$
Then, $\mathcal{F}$ has a unique fixed point, i.e., there exists a unique element $y^{\ast}\in C([0,T];\H)$ such that $\mathcal{F}(y^{\ast})=y^{\ast}$.
\end{proposition}

\section{Enhanced Gronwall's inequalities}\label{sec1}
In this section we present two new versions of Gronwall's inequality. We will see later that these two inequalities play a key role in obtaining the existence of solutions for integro-differential sweeping processes of Volterra type.
\begin{Theorem}[Enhanced Gronwall's Inequality I]\label{TeoI}
Let $I:=[T_0,T]$, and let $\rho\colon I \to \mathbb{R}$ be a nonnegative absolutely continuous function. Let $K_1, K_2, \varepsilon \colon I \to \mathbb{R}_+$, and  $K_3 \colon I\times I \to \mathbb{R}_+$ be nonnegative measurable functions such 
$$
t\mapsto K_1(t) \textrm{ and } t\mapsto K_2(t)\int_{T_0}^t K_3(t,s)ds \textrm{ are } integrable.
$$
Suppose that
$$
\dot{\rho}(t)\leq \varepsilon(t)+K_1(t)\rho(t)+K_2(t)\int_{T_0}^t K_3(t,s)\rho(s)ds \quad \textrm{ for a.e. } t\in I.
$$
Then,  one has
$$
\rho(t)\leq \rho(T_0)\exp\left(\int_{T_0}^t \gamma(s)ds\right)+\int_{T_0}^t \varepsilon(s)\exp\left(\int_s^t \gamma(\tau)d\tau\right)ds \quad \textrm{ for all }t\in I,
$$
where  $\gamma(t):=K_1(t)+K_2(t)\int_{T_0}^t K_3(t,s)ds$.
\end{Theorem}
\begin{proof} Let us consider the non-decreasing function:
$$
\vartheta(t):=\rho(T_0)+\int_{T_0}^t \varepsilon(s)ds+\int_{T_0}^tK_1(s)\rho(s)ds+\int_{T_0}^tK_2(s)\int_{T_0}^{s} K_3(s,\tau)\rho(\tau)d\tau ds.
$$
Then, by assumption, $\rho(t)\leq  \vartheta(t)$ for all $t\in I$. Therefore, for a.e. $t\in I$, one has
\begin{equation*}
\begin{aligned}
\dot{\vartheta}(t)&=  \varepsilon(t)+K_1(t)\rho(t)+K_2(t)\int_{T_0}^t K_3(t,s)\rho(s)ds\\
&\leq  \varepsilon(t)+(K_1(t)+K_2(t)\int_{T_0}^t K_3(t,s)ds)\vartheta(t)\\
&=\varepsilon(t)+\gamma(t)\vartheta(t),
\end{aligned}
\end{equation*}
where we have used that  $\rho(t)\leq  \vartheta(t)$ for all $t\in I$ and the function $t \mapsto \vartheta(t)$ is nondecreasing. Hence, by applying the classical Gronwall's inequality we get that 
$$
\vartheta(t)\leq \vartheta(T_0)\exp\left(\int_{T_0}^t \gamma(s)ds\right)+\int_{T_0}^t \varepsilon(s)\exp\left(\int_s^t \gamma(\tau)d\tau\right)ds \quad \textrm{ for all }t\in I.
$$
Finally, by using that $\vartheta(T_0)=\rho(T_0)$ and $\rho(t)\leq \vartheta(t)$ for all $t\in I$, we obtain the result.
\end{proof}

\begin{Theorem}[Enhanced Gronwall's Inequality II]\label{TeoII}
Let $I:=[T_0,T]$, and let $\rho\colon I \to \mathbb{R}$ be a nonnegative absolutely continuous function.  Let $\varepsilon, K_1, K_2, K_3, K_4 \colon I \to \mathbb{R}_+$ be nonnegative measurable functions such 
$$
t\mapsto \varepsilon(t), t\mapsto K_1(t), t\mapsto K_2(t) \textrm{ and } t\mapsto K_3(t)\int_{T_0}^t K_4(s)ds \textrm{ are } integrable.
$$
Suppose that for a.e. $t\in I$ one has
$$
\dot{\rho}(t)\leq \varepsilon(t)+K_1(t)\sqrt{\rho(t)}+K_2(t)\rho(t)+K_3(t)\sqrt{\rho(t)}\int_{T_0}^t K_4(s)\sqrt{\rho(s)}ds.
$$
Then, the following assertions hold:
\begin{itemize}
\item[(a)] For all $t\in I$, one has 
$$
\rho(t)\leq \rho(T_0)\exp\left(\int_{T_0}^t (\gamma(s)+K_1(s))ds\right)+\int_{T_0}^t (\varepsilon(s)+K_1(s))\exp\left(\int_s^t (\gamma(\tau)+K_1(\tau))d\tau\right)ds,
$$
where $\gamma(t):=(K_2(t)+K_3(t)\int_{T_0}^t K_4(s)ds)$. 
\item[(b)] If, in addition $\varepsilon(t)\equiv 0$, then for all $t\in I$, one has 
$$
\sqrt{\rho(t)}\leq \sqrt{\rho(T_0)}\exp\left(\frac{1}{2}\int_{T_0}^t \gamma(s)ds\right)+\frac{1}{2}\int_{T_0}^t K_1(s)\exp\left(\frac{1}{2}\int_s^t \gamma(\tau)d\tau\right)ds.
$$
\end{itemize}

\end{Theorem}
\begin{proof} Let us consider the non-decreasing function $\vartheta(t)$ such that $\vartheta(0)=\rho(T_0)$ and 
$$
\dot{\vartheta}(t):=\varepsilon(t)+K_1(t)\sqrt{\rho(t)}+K_2(t)\rho(t)+K_3(t)\sqrt{\rho(t)}\int_{T_0}^t K_4(s)\sqrt{\rho(s)}ds.
$$
Then, by assumption, $\rho(t)\leq  \vartheta(t)$ for all $t\in I$. Therefore, for a.e. $t\in I$, one has
\begin{equation*}
\begin{aligned}
\dot{\vartheta}(t)&=  \varepsilon(t)+K_1(t)\sqrt{\rho(t)}+K_2(t)\rho(t)+K_3(t)\sqrt{\rho(t)}\int_{T_0}^t K_4(s)\sqrt{\rho(s)}ds\\
&\leq  \varepsilon(t)+K_1(t)\sqrt{\vartheta(t)}+(K_2(t)+K_3(t)\int_{T_0}^t K_4(s)ds)\vartheta(t)\\
&=\varepsilon(t)+K_1(t)\sqrt{\vartheta(t)}+\gamma(t)\vartheta(t),
\end{aligned}
\end{equation*}
where we have used that $t \mapsto \vartheta(t)$ is nondecreasing. Moreover, by noting that $\sqrt{\vartheta(t)}\leq \vartheta(t) +1$ for all $t\in I$,  we obtain that  
\begin{equation*}
\begin{aligned}
\dot{\vartheta}(t)\leq  (\varepsilon(t)+K_1(t))+(\gamma(t)+K_1(t))\vartheta(t) \textrm{ a.e. } t\in I,
\end{aligned}
\end{equation*}
which, by the classical Gronwall's inequality,  implies that
$$
\vartheta(t)\leq \vartheta(T_0)\exp\left(\int_{T_0}^t (\gamma(t)+K_1(s))ds\right)+\int_{T_0}^t (\varepsilon(s)+K_1(t))\exp\left(\int_s^t (\gamma(\tau)+K_1(\tau))d\tau\right)ds \quad t\in I.
$$
Then, by using that $\vartheta(T_0)=\rho(T_0)$ and $\rho(t)\leq \vartheta(t)$ for $t\in I$, the first assertion is proved. To prove $(b)$, assume that $\varepsilon(t)\equiv 0$ and set $\sigma(t):=\sqrt{\rho(t)}$. Then, by virtue of Lemma \ref{lemma-derivative},  we obtain that
$$
\dot{\sigma}(t)\leq \frac{1}{2}K_1(t)+\frac{1}{2}K_2(t)\sigma(t)+\frac{1}{2}K_3(t)\int_{T_0}^t K_4(s)\sigma(s)ds.
$$
Hence, applying Theorem \ref{TeoI}, we obtain the result. 
\end{proof}

\section{Integro-differential sweeping process of Volterra type}\label{sec2}

In this section, we will use enhanced Gronwall's inequalities proved in Section \ref{sec1} to obtain the existence and uniqueness of solutions for the integro-differential sweeping process of Volterra type:  
\begin{equation}\label{SP}\tag{$\mathcal{SP}(x_0,z,f,g)$}
\left\{
\begin{aligned}
\dot{x}(t)& \in -N_{C(t)+z(t)}(x(t))+f(t,x(t))+\int_{T_0}^t g(t,s,x(s))ds &  \textrm{ a.e. } t\in I,\\
x(T_0)&=x_0\in C(T_0),
\end{aligned}
\right.
\end{equation}
where $z\colon I\to \H$ is an absolutely continuous function, $C(t)$ is a nonempty, closed and $\rho$-uniformly prox-regular set for all $t\in I$ and $f\colon I\times \H \to \H$ and $g\colon I\times I\to \H$ are functions satisfying assumptions \ref{Hf} and \ref{Hg} below.

\paragraph{Assumptions on the data:} We consider the following assumptions on the data for \eqref{SP}.
\begin{enumerate}
\item[\namedlabel{HC}{$(\mathcal{H}_C)$}] The set-valued map $C\colon I\rightrightarrows \H$ has nonempty, closed and $\rho$-uniformly prox-regular values for some constant $\rho\in ]0,+\infty]$. Moreover, there exist an absolutely continuous function $v\colon I\to \mathbb{R}$ such that  
\begin{equation*}
C(t)\subset C(s)+|v(t)-v(s)| \mathbb{B} \textrm{ for all } s,t\in I.
\end{equation*}
\end{enumerate}

\begin{enumerate}
    \item[\namedlabel{Hf}{$(\mathcal{H}^f)$}]  The function $f\colon I\times \H\to \H$ satisfies
    \begin{enumerate}
        %\item[\namedlabel{H1g}{$(\mathcal{H}_1^g)$}] For each $v\in \H$, $g(\cdot,v)$ is measurable.
        \item For each $x\in \H$, the map $t \mapsto f(t,x)$ is measurable.
        \item For all $r>0$, there exists an integrable function $\kappa_{r}\colon I \to \mathbb{R}_+$ such that        $$
        \Vert f(t,x)-f(t,y)\Vert \leq \kappa_{r}(t)\Vert x-y\Vert \textrm{ for all }  t\in I \textrm{ and } x,y\in r\mathbb{B}.
        $$

        \item There exists a nonnegative integrable function $\beta\colon I \to \mathbb{R}$ such that 
        $$
        \Vert f(t,x)\Vert \leq \beta(t)(1+\Vert x\Vert) \textrm{ for all } t \in I \textrm{ and } x\in \H. 
        $$
    \end{enumerate}
\end{enumerate}

\begin{enumerate}
    \item[\namedlabel{Hg}{$(\mathcal{H}^g)$}]  The function $g\colon I \times I\times \H\to \H$ satisfies
    \begin{enumerate}
        \item For each $x\in \H$, the map $(t,s) \mapsto g(t,s,x)$ is measurable.
        \item For all $r>0$, there exists an integrable function $\mu_{r}\colon I\to \mathbb{R}_+$ such that for all $(t,s)\in D$ 
        $$
        \Vert g(t,s,x)-g(t,s,y)\Vert \leq \mu_{r}(t)\Vert x-y\Vert \textrm{ for all } x,y\in r\mathbb{B}.
        $$
        Here $D:=\{(t,s)\in I\times I \colon s\leq t\}$. 
        \item There exists a nonnegative integrable function $\sigma\colon D\to \mathbb{R}$ such that 
        $$
        \Vert g(t,s,x)\Vert \leq \sigma(t,s)(1+\Vert x\Vert) \textrm{ for all } (t,s)\in D \textrm{ and } x\in \H. 
        $$
    \end{enumerate}
\end{enumerate}
In order to prove the existence of solutions, we will apply a fixed point argument (see Proposition \ref{History}) to the classical sweeping process. Hence, the following result will be needed (see \cite[Proposition~1]{MR2179241}).
\begin{Lemma}\label{Lemma-1} Let $\H$ be a real Hilbert space, and suppose that $C(\cdot)$ satisfies \ref{HC}. Let $h\colon I\to \H$ be a single-valued integrable mapping. Then, for any $x_0\in C(T_0)$ there exists a unique absolutely continuous solution $x(\cdot)$ for the following differential inclusion:
\begin{equation*}
\left\{
\begin{aligned}
\dot{x}(t)&\in -N(C(t);x(t))+h(t) \quad \textrm{ a.e. } t\in I\\
x(T_0)&=x_0.
\end{aligned}
\right.
\end{equation*}
Moreover, $x(\cdot)$ satisfies the following inequality:
$$
\Vert \dot{x}(t)-h(t)\Vert \leq \Vert h(t)\Vert +\vert \dot{v}(t)\vert \quad \textrm{ a.e. } t\in I.
$$
\end{Lemma}
Now, we present the main result of this section, that is, the well-posedness for the nonconvex integro-differential sweeping process \eqref{SP}.  It is worth emphasizing that the existence of solutions for \eqref{SP} was obtained in \cite{MR4492538,MR4421896} using an approximation scheme and different Gronwall inequalities. In this section, our contribution is to get the existence of solutions through a fixed point argument and enhanced versions of Gronwall's inequality obtained in Section \ref{sec1}.  We refer to \cite[Theorem~4.1]{MR4614261} for a related fixed point argument for ball compact moving sets.

\begin{Theorem}\label{Existence} Let $z$ be an absolutely continuous function and assume, in addition to \ref{HC}, that \ref{Hf} and \ref{Hg} hold.  Then, for any $x_0\in C(T_0)$, there exists a unique absolutely continuous solution $x(\cdot)$ of \eqref{SP}. \newline  Moreover, the following bounds hold
\begin{equation*}
\begin{aligned}
\Vert x(t)\Vert &\leq r(t):=\Vert x_0\Vert \exp\left( \int_{T_0}^t \gamma(s)ds\right)+\int_{T_0}^t (\gamma(s)+\vert \dot{v}(s)\vert+\Vert \dot{z}(t)\Vert)\exp\left(\int_s^t \gamma(\tau)d\tau \right)ds &\textrm{ for all } t\in I,\\
\Vert \dot{x}(t)\Vert & \leq \theta(t):=\gamma(t)(1+r(t))+\vert \dot{v}(t)\vert +\Vert \dot{z}(t)\Vert &\textrm{ for a.e. } t\in I,
\end{aligned}
\end{equation*}
where $\gamma(t):=2\beta(t)+2\int_{T_0}^t \sigma(t,s)ds$.
\end{Theorem}
\begin{proof}
Let us consider the integrable function 
\begin{equation}\label{def-phi}
\varphi(t):=\max\{1,\kappa_{r(T)}(t),\mu_{r(T)}(t),\beta(t),\int_{T_0}^t \sigma(t,s)ds\}.
\end{equation}
 The map $t\mapsto \int_{T_0}^t \varphi(s)ds$ from $I$ to $J:=[0,\int_{T_0}^T \varphi(s)ds]$ is continuous and strictly increasing. Let $\phi$ be its inverse and define
$$
\tilde{f}(t,x):=\frac{1}{\varphi(\phi(t))}f(\phi(t),x) \textrm{ and } \tilde{g}(t,s,x)=\frac{1}{\varphi(\phi(t))}\frac{1}{\varphi(\phi(s))}g(\phi(t),\phi(s),x).
$$ 
It is clear that $\tilde{f}$ and $\tilde{g}$ satisfy \ref{Hf} and \ref{Hg}, respectively. Moreover, for all $x,y\in r(T)\mathbb{B}$, one has
\begin{equation}\label{cotas-fg}
\begin{aligned}
\Vert \tilde{f}(t,x)-\tilde{f}(t,y)\Vert \leq \Vert x-y\Vert \textrm{ and }  \Vert \tilde{g}(t,s,x)-\tilde{g}(t,s,y)\Vert \leq \Vert x-y\Vert. 
\end{aligned}
\end{equation}
\noindent  Given $y\in C(J;\H)$, let us consider  $w(\cdot)$ be the unique solution (see Lemma \ref{Lemma-1}) to the problem:
\begin{equation}\label{Sweeping-bi}
\left\{
\begin{aligned}
\dot{w}(t)&\in -N_{C(\phi(t))+z(\phi(t))}(w(t))+h(y)(t) \quad \textrm{ a.e. } t\in J,\\
 w(0)&=x_0,
\end{aligned} 
\right.
\end{equation}
where $h(y)(t):=\tilde{f}(t,\operatorname{proj}_{r(\phi(t))\mathbb{B}}(y(t)))+\int_{0}^t \tilde{g}(t,s,\operatorname{proj}_{r(\phi(s))\mathbb{B}}(y(s)))ds$. According to Lemma \ref{Lemma-1},  \ref{Hf}, and \ref{Hg}, for a.e. $t\in J$ one has
\begin{equation*}
\begin{aligned}
\Vert \dot{w}(t)-h(y)(t)\Vert &\leq \Vert h(y)(t)\Vert +\vert \dot{v}(\phi(t))\vert +\Vert \dot{z}(\phi(t))\Vert\\
&\leq \frac{\beta(\phi(t))}{\varphi(\phi(t))}(1+r(\phi(t)))+\int_{0}^{t} \frac{\sigma(\phi(t),\phi(s))}{\varphi(\phi(t))\varphi(\phi(s))}(1+r(\phi(s)))ds+\vert \dot{v}(\phi(t))\vert+\Vert \dot{z}(\phi(t))\Vert\\
&\leq \left[\frac{\beta(\phi(t))}{\varphi(\phi(t))}+\int_{0}^{t} \frac{\sigma(\phi(t),\phi(s))}{\varphi(\phi(t))\varphi(\phi(s))}ds\right] (1+r(\phi(t)))+\vert \dot{v}(\phi(t))\vert+\Vert \dot{z}(\phi(t))\Vert\\
&=\left[\frac{\beta(\phi(t))}{\varphi(\phi(t))}+\frac{\int_{T_0}^{\phi(t)} \sigma(\phi(t),s)ds}{{\varphi(\phi(t))}}\right] (1+r(\phi(t)))+\vert \dot{v}(\phi(t))\vert+\Vert \dot{z}(\phi(t))\Vert\\
&\leq \eta(t):=2(1+r(\phi(t)))+\vert \dot{v}(\phi(t))\vert+\Vert \dot{z}(\phi(t))\Vert,
\end{aligned}
\end{equation*}
where we have used that $t\mapsto r(\phi(t))$ is nondecreasing and definition \eqref{def-phi}. \newline
Next, we consider the operator $\mathcal{F}\colon C(J;\H)\to C(J;\H)$ defined as $\mathcal{F}(y):=w$, where $w(\cdot)$ is the unique solution of  the problem \eqref{Sweeping-bi}.\newline
\textbf{Claim 1:} The operator $\mathcal{F}$ has a unique fixed point.\newline 
\emph{Proof of Claim 1}: 
Let us consider $V(t):=\frac{1}{2}\Vert w_1(t)-w_2(t)\Vert^2$, where $w_i:=\mathcal{F}(y_i)$, for $i=1,2$. Then, for a.e. $t\in J$, we obtain that
\begin{equation*}
\begin{aligned}
\dot{V}(t)&=\langle w_1(t)-w_2(t),\dot{w}_1(t)-\dot{w}_2(t)\rangle \\
&=\langle (w_1(t)-z(\phi(t)))-(w_2(t)-z(\phi(t))),(\dot{w}_1(t)-h(y_1)(t))-(\dot{w}_2(t)-h(y_2)(t))\rangle\\
&+\langle w_1(t)-w_2(t),h(y_1)(t)-h(y_2)(t)\rangle\\
&\leq \frac{\eta(t)}{\rho}\Vert w_1(t)-w_2(t)\Vert^2+\Vert w_1(t)-w_2(t)\Vert\cdot\Vert h(y_1)(t)-h(y_2)(t)\Vert,
\end{aligned}
\end{equation*}
where we have used Proposition \ref{prox-regularity} and that for $i=1,2,$ one has
$$
\frac{w_i(t)-h(y_i)(t)}{\eta(t)}\in -N_{C(\phi(t))}(w_i(t)-z(\phi(t)))\cap \mathbb{B} \textrm{ for a.e. } t\in J.
$$
Next, if we consider $W(t):=\Vert w_1(t)-w_2(t)\Vert$, according to Lemma \ref{lemma-derivative}, we obtain that
\begin{equation*}
\dot{W}(t)\leq  \frac{\eta(t)}{\rho}W(t)+\Vert h(y_1)(t)-h(y_2)(t)\Vert  \textrm{ for a.e. } t\in J.
\end{equation*}
On the other hand, by virtue of \eqref{cotas-fg} and the Lipschitzianity of constant 1 of the map $x\mapsto \operatorname{proj}_{r(t)}(x)$, we obtain that for a.e. $t\in J$
\begin{equation*}
\begin{aligned}
\Vert h(y_1)(t)-h(y_2)(t)\Vert &\leq \Vert \tilde{f}(t,\operatorname{proj}_{r(t)\mathbb{B}}(y_1(t)))-\tilde{f}(t,\operatorname{proj}_{r(t)\mathbb{B}}(y_2(t)))\Vert \\
&+\int_{0}^t \Vert \tilde{g}(t,s,\operatorname{proj}_{r(s)\mathbb{B}}(y_1(s)))-\tilde{g}(t,s,\operatorname{proj}_{r(s)\mathbb{B}}(y_2(s)))\Vert ds\\
&\leq \Vert y_1(t)-y_2(t)\Vert +\int_{0}^t \Vert y_1(s)-y_2(s)\Vert  ds.
\end{aligned}
\end{equation*}
Therefore, for a.e.  $t\in J$, 
\begin{equation*}
\begin{aligned}
\frac{d}{dt}\left( W(t)\exp\left(-\int_0^t \frac{\eta(s)}{\rho}ds\right)\right)&\leq  \exp\left(-\int_0^t \frac{\eta(s)}{\rho}ds\right) \left( \Vert y_1(t)-y_2(t)\Vert +\int_{0}^t \Vert y_1(s)-y_2(s)\Vert  ds\right)\\
&\leq e^t \left( \Vert y_1(t)-y_2(t)\Vert +\int_{0}^t \Vert y_1(s)-y_2(s)\Vert  ds\right)\\
&=\frac{d}{dt} \left( e^t \int_{0}^t \Vert y_1(s)-y_2(s)\Vert  ds\right),
\end{aligned}
\end{equation*}
which implies that
$$
W(t)\leq \exp\left(\int_0^t \frac{\eta(s)}{\rho}ds+1\right)\int_{0}^t \Vert y_1(s)-y_2(s)\Vert  ds \quad  \textrm{ for all } t\in J.
$$
Hence, the operator $\mathcal{F}$ is history-dependent. Therefore, by virtue of Proposition \ref{History}, the operator $\mathcal{F}$ has a unique fixed point $\bar{w}\in C(J;\H)$, which proves Claim 1. \newline
On the other hand, it is routine to prove that $x(t):=\bar{w}\left(\int_{T_0}^t \varphi(s)ds\right)$ solves the problem: 
\begin{equation}\label{Sweeping-bi2}
\left\{
\begin{aligned}
\dot{x}(t)&\in -N_{C(t)+z(t)}(x(t))+f(t,\operatorname{proj}_{r(t)\mathbb{B}}(x(t)))+\int_{T_0}^t {g}(t,s,\operatorname{proj}_{r(s)\mathbb{B}}(x(s)))ds \quad \textrm{ a.e. } t\in I,\\
 x(T_0)&=x_0.
\end{aligned} 
\right.
\end{equation}
\textbf{Claim 2}: For all $t\in I$, $\Vert x(t)\Vert \leq r(t)$. \newline
\noindent \emph{Proof of Claim 2:} According to Lemma \ref{Lemma-1}, for a.e. $t\in I$, one has
\begin{equation*}
\begin{aligned}
\Vert \dot{x}(t)\Vert & \leq 2\Vert f(t,\operatorname{proj}_{r(t)\mathbb{B}}(x(t)))\Vert +2\int_{T_0}^t \Vert g(t,s,\operatorname{proj}_{r(s)\mathbb{B}}(x(s)))\Vert ds+\vert \dot{v}(t)\vert +\Vert \dot{z}(t)\Vert\\
&\leq 2\beta(t)(1+\Vert \operatorname{proj}_{r(t)\mathbb{B}}(x(t))\Vert)+2\int_{T_0}^t \sigma(t,s)\Vert \operatorname{proj}_{r(s)\mathbb{B}}(x(s))\Vert ds+ \vert \dot{v}(t)\vert +\Vert \dot{z}(t)\Vert\\
&\leq 2\beta(t)(1+\Vert x(t)\Vert)+2\int_{T_0}^t \sigma(t,s)(1+\Vert x(s)\Vert) ds+ \vert \dot{v}(t)\vert +\Vert \dot{z}(t)\Vert,
\end{aligned}
\end{equation*}
where we have used that $\Vert \operatorname{proj}_{r(t)\mathbb{B}}(x(t))\Vert \leq  \Vert x(t)\Vert$. Therefore, according to Theorem \ref{TeoI}, we obtain that $\Vert x(t)\Vert \leq r(t)$ for all $t\in I$, which proves Claim 2. \newline 
\textbf{Claim 3}:  The trajectory $x$ solves \eqref{SP} and for a.e. $t\in I$, $\Vert \dot{x}(t)\Vert \leq \theta(t)$.\\
\noindent \emph{Proof of Claim 3:} Since $\Vert x(t)\Vert \leq r(t)$ for all $t\in I$, it follows from \eqref{Sweeping-bi2} that $x$ solves \eqref{SP}. Moreover, from Lemma \ref{Lemma-1}, it follows that 
\begin{equation*}
\begin{aligned}
\Vert \dot{x}(t)\Vert & \leq 2\Vert f(t,x(t))\Vert +2\int_{T_0}^t \Vert g(t,s,x(s))\Vert ds+\vert \dot{v}(t)\vert +\Vert \dot{z}(t)\Vert\\
&\leq 2\beta(t)(1+ r(t))+2\int_{T_0}^t \sigma(t,s)(1+ r(s)) ds+ \vert \dot{v}(t)\vert +\Vert \dot{z}(t)\Vert\\
&\leq\gamma(t)(1+r(t))+ \vert \dot{v}(t)\vert +\Vert \dot{z}(t)\Vert,
\end{aligned}
\end{equation*}
which proves Claim 3.\newline 
\textbf{Claim 4}:  The trajectory $x$ is the unique solution of \eqref{SP}.\newline 
\noindent \emph{Proof of Claim 4:} Let $x_1$ and $x_2$ be two solutions of \eqref{SP}. Consider  
$V(t):=\frac{1}{2}\Vert w_1(t)-w_2(t)\Vert^2$, where $w_i(t):=x_i(\phi(t))$, for $i=1,2$. Similarly to the proof of Claim 1, we can prove the existence of a constant $\mathcal{C}\geq 0$ such that 
$$
\Vert w_1(t)-w_2(t)\Vert \leq \mathcal{C}\int_{0}^t \Vert w_1(s)-w_2(s)\Vert ds \textrm{ for all } t\in J.
$$
Then, by using that $w_1(0)=w_2(0)$ and the above inequality, we obtain that $w_1\equiv w_2$. Therefore, $x_1\equiv x_2$, which proves Claim 4.  
\end{proof}

\subsection{Integro-differential sweeping processes driven by a fixed convex set}

In this section, we consider the particular case of  $C(t)\equiv C$, where $C$ is a nonempty, closed and convex set of a Hilbert space $\H$. Hence, we are interested in the problem
\begin{equation}\label{SP-M}
\left\{
\begin{aligned}
\dot{x}(t)& \in -N_{C}(x(t))+f(t,x(t))+\int_{T_0}^t g(t,s,x(s))ds &  \textrm{ a.e. } t\in I,\\
x(T_0)&=x_0\in C.
\end{aligned}
\right.
\end{equation}
In the present case, the operator $x\mapsto N_C(x)$ is maximal monotone, making possible to obtain more information about the behavior of the solutions.  We recall the concept of slow solution introduced in \cite{MR717499} and widely used in differential inclusions governed by maximal monotone operators (see, e.g., \cite{MR348562}).
\begin{definition}[Slow solution]
We say that $x(\cdot)$ is a \emph{slow} solution of \eqref{SP-M} if  $x(T_0)=x_0$ and 
$$
\dot{x}(t)=-\bar{\nu}(t)+f(t,x(t))+\int_{T_0}^t g(t,s,x(s))ds \quad \textrm{ a.e. } t\in I,
$$
where
$$
\bar{\nu}(t):=\operatorname{argmin}_{\nu \in N_C(x(t))}\Vert f(t,x(t))+\int_{T_0}^t g(t,s,x(s))ds-\nu\Vert.
$$
\end{definition}
The next result asserts that the unique solution of \eqref{SP-M} is slow.
\begin{Theorem}\label{Existence-2} Assume, in addition to \ref{HC}, that \ref{Hf} and \ref{Hg} hold.  Then, for any $x_0\in C$, there exists a unique absolutely continuous solution $x(\cdot)$ of \eqref{SP-M}. Moreover, $x(\cdot)$ is a slow solution of \eqref{SP-M} and the following bounds hold
\begin{equation*}
\begin{aligned}
\Vert x(t)\Vert &\leq r(t):=\Vert x_0\Vert \exp\left( \int_{T_0}^t \gamma(s)ds\right)+\int_{T_0}^t \gamma(s)\exp\left(\int_s^t \gamma(\tau)d\tau \right)ds &\textrm{ for all } t\in I,\\
\Vert \dot{x}(t)\Vert & \leq \theta(t):=\gamma(t)(1+r(t)) &\textrm{ for a.e. } t\in I,
\end{aligned}
\end{equation*}
where $\gamma(t):=2\beta(t)+2\int_{T_0}^t \sigma(t,s)ds$.
\end{Theorem}
\begin{proof}
It follows from \cite[Proposition~3.4]{MR348562} and  Theorem \ref{Existence}.
\end{proof}

\section{Continuous dependence of solutions for the integro-differential sweeping process}\label{sec3}
In this section, we prove the continuous dependence of solutions for the integro-differential sweeping process of Volterra type. Our approach is based on enhanced Gronwall inequalities established in Section \ref{sec1}. \newline 
\noindent Given a nonnegative integrable function $R$, we consider the set
$$
\mathcal{B}_R:=\{z\in AC(I;\H)\colon \Vert \dot{z}(t)\Vert \leq R(t) \textrm{ a.e. } t\in I\}.
$$
In what follows, given $r>0$, we denote
\begin{equation*}
d_r(f_1,f_2)(t):=\sup_{x\in r\mathbb{B}}\Vert f_1(t,x)-f_2(t,x)\Vert \textrm{ and } D_r(g_1,g_2)(t,s):=\sup_{x\in r\mathbb{B}}\Vert g_1(t,s,x)-g_2(t,s,x)\Vert.
\end{equation*}

\begin{Theorem}\label{Well-posedness}
 Assume that \ref{HC} holds.  Let $x_0^1, x_0^2\in C(T_0)$,  $z_1, z_2\in \mathcal{B}_R$, and assume that  $f_1$, $f_2$, and $g_1$, $g_2$ satisfy \ref{Hf} and \ref{Hg}, respectively. Let $x_i$ be the unique solution of $\mathcal{SP}(x_0^i,z_i,f_i,g_i)$ for $i=1,2$. Then, for all $t\in I$, one has
\begin{equation*}
\begin{aligned}
\Vert x_1(t)-x_2(t)\Vert^2 \leq \Vert x_0^1-x_0^2\Vert^2 \exp\left(\int_{T_0}^t (\delta(s)+\Delta(s))ds\right)+2\int_{T_0}^t (\varepsilon(s)+\Delta(s))\exp\left(\int_s^t (\delta(\tau)+\Delta(\tau))d\tau\right)ds,
\end{aligned}
\end{equation*}
where 
\begin{equation*}
\begin{aligned}
\Delta(t)&:=\sqrt{2} d_{r(T)}(f_1,f_2)(t)+\sqrt{2}\int_{T_0}^tD_{r(T)}(g_1,g_2)(t,s)ds\\
\delta(t)&:=2\frac{\nu(t)}{\rho}+2\kappa_{r(T)}+2\int_{T_0}^t \mu_{r(T)}(s)ds\\
\varepsilon(t)&:=2\nu(t)\Vert z_1(t)-z_2(t)\Vert \textrm{ and } \nu(t):=\frac{\gamma(t)}{2}(1+r(t))+\vert \dot{v}(t)\vert +R(t).
\end{aligned}
\end{equation*}
\end{Theorem}
\begin{proof} For $i=1,2$, let us consider 
$$
h_i(t):=f_i(t,x_i(t))+\int_{T_0}^t g_i(t,s,x_i(s))ds.
$$
Then, according to Theorem \ref{Existence}, for $i=1,2$, one has
\begin{equation}\label{proximal3}
\Vert \dot{x}_i(t)-h_i(t)\Vert \leq \nu(t):=\frac{\gamma(t)}{2}(1+r(t))+\vert \dot{v}(t)\vert +R(t),
\end{equation}
where $$r(t):=\Vert x_0\Vert \exp\left( \int_{T_0}^t \gamma(s)ds\right)+\int_{T_0}^t (\gamma(s)+\vert \dot{v}(s)\vert+R(t))\exp\left(\int_s^t \gamma(\tau)d\tau \right)ds.$$
Hence, for $i=1,2$, one has
$$
\frac{\dot{x}_i(t)-h_i(t)}{\nu(t)}\in -N_{C(t)}(x_i(t)-z_i(t))\cap \mathbb{B} \textrm{ for a.e. }t\in I.
$$
Next, let us consider $V(t):=\frac{1}{2}\Vert x_1(t)-x_2(t)\Vert^2$. Then, for a.e. $t\in I$, we obtain that
\begin{equation*}
\begin{aligned}
\dot{V}(t)&=\langle x_1(t)-x_2(t),\dot{x}_1(t)-\dot{x}_2(t)\rangle \\
&=\langle (x_1(t)-z_1(t))-(x_2(t)-z_2(t)),(\dot{x}_1(t)-h_1(t))-(\dot{x}_2(t)-h_2(t))\rangle\\
&+\langle z_1(t)-z_2(t),(\dot{x}_1(t)-h_1(t))-(\dot{x}_2(t)-h_2(t))\rangle\\
&+\langle x_1(t)-x_2(t),h_1(t)-h_2(t)\rangle\\
&\leq \frac{\nu(t)}{\rho}\Vert x_1(t)-x_2(t)\Vert^2+2\nu(t)\Vert z_1(t)-z_2(t)\Vert+\Vert x_1(t)-x_2(t)\Vert \cdot \Vert h_1(t)-h_2(t)\Vert\\
&=\frac{2\nu(t)}{\rho}V(t)+2\nu(t)\Vert z_1(t)-z_2(t)\Vert+\sqrt{2V(t)} \cdot \Vert h_1(t)-h_2(t)\Vert,
\end{aligned}
\end{equation*}
where we have used Proposition \ref{prox-regularity} and inequality \eqref{proximal3}. On the other hand, for all $t\in I$, one has
\begin{equation*}
\begin{aligned}
\Vert h_1(t)-h_2(t)\Vert &\leq \Vert f_1(t,x_1(t))-f_2(t,x_2(t))\Vert +\int_{T_0}^t \Vert g_1(t,s,x_1(s))-g_2(t,s,x_2(s))\Vert ds\\
&\leq d_{r(T)}(f_1,f_2)(t)+\kappa_{r(T)}(t)\Vert x_1(t)-x_2(t)\Vert \\
&+\int_{T_0}^tD_{r(T)}(g_1,g_2)(t,s)ds+\int_{T_0}^t \mu_{r(T)}(s)\Vert x_1(s)-x_2(s)\Vert ds\\
&= d_{r(T)}(f_1,f_2)(t)+\kappa_{r(T)}(t)\sqrt{2V(t)} +\int_{T_0}^tD_{r(T)}(g_1,g_2)(t,s)ds+\int_{T_0}^t \mu_{r(T)}(s)\sqrt{2V(s)} ds\end{aligned}
\end{equation*}
Hence, by using the above inequalities, we obtain that for a.e. $t\in I$, one has
\begin{equation}\label{inequality-V}
\begin{aligned}
\dot{V}(t)&\leq 2\nu(t)\Vert z_1(t)-z_2(t)\Vert +\sqrt{2}\left( d_{r(T)}(f_1,f_2)(t)+\int_{T_0}^tD_{r(T)}(g_1,g_2)(t,s)ds\right)\sqrt{V(t)}\\
&+2\left(\frac{\nu(t)}{\rho}+\kappa_{r(T)}\right)V(t)+2\sqrt{V(t)}\int_{T_0}^t \mu_{r(T)}(s)\sqrt{V(s)}ds.
\end{aligned}
\end{equation}
Therefore, by the enhanced Gronwall's inequality II (Theorem \ref{TeoII}), we obtain that for all $t\in I$, one has
\begin{equation*}
\begin{aligned}
V(t)\leq V(T_0)\exp\left(\int_{T_0}^t (\delta(s)+\Delta(s))ds\right)+\int_{T_0}^t (\varepsilon(s)+\Delta(s))\exp\left(\int_s^t (\delta(\tau)+\Delta(\tau))d\tau\right)ds,
\end{aligned}
\end{equation*}
which proves the result.
\end{proof}
\noindent It is worth noting that when $z_1$ and $z_2$ are equal, the conclusions of the above theorem can be improved.
\begin{Theorem} Assume that \ref{HC} holds. Let $x_0^1, x_0^2\in C(T_0)$ and $z\in \mathcal{B}_R$. Suppose that  $f_1$, $f_2$, and $g_1$, $g_2$ satisfy \ref{Hf} and \ref{Hg}, respectively. Let $x_i$ be the unique solution of $\mathcal{SP}(x_0^i,z,f_i,g_i)$ for $i=1,2$. Then, for all $t\in I$, one has
\begin{equation*}
\begin{aligned}
\Vert x_1(t)-x_2(t)\Vert \leq \Vert x_0^1-x_0^2\Vert \exp\left(\int_{T_0}^t \delta(s)ds\right)+\int_{T_0}^t \Delta(s)\exp\left(\int_s^t \delta(\tau)d\tau\right)ds,
\end{aligned}
\end{equation*}
where 
\begin{equation*}
\begin{aligned}
\Delta(t)&:=\sqrt{2} d_{r(T)}(f_1,f_2)(t)+\sqrt{2}\int_{T_0}^tD_{r(T)}(g_1,g_2)(t,s)ds,\\
\delta(t)&:=2\frac{\nu(t)}{\rho}+2\kappa_{r(T)}(t)+2\int_{T_0}^t \mu_{r(T)}(s)ds,\\
\nu(t)&:=\frac{\gamma(t)}{2}(1+r(t))+\vert \dot{v}(t)\vert +R(t).
\end{aligned}
\end{equation*}
\end{Theorem}
\begin{proof}
From \eqref{inequality-V}, we obtain that 
\begin{equation*}
\begin{aligned}
\dot{V}(t)&\leq \sqrt{2}\left( d_{r(T)}(f_1,f_2)(t)+\int_{T_0}^tD_{r(T)}(g_1,g_2)(t,s)ds\right)\sqrt{V(t)}\\
&+2\left(\frac{\nu(t)}{\rho}+\kappa_{r(T)}\right)V(t)+2\sqrt{V(t)}\int_{T_0}^t \mu_{r(T)}(s)\sqrt{V(s)}ds.
\end{aligned}
\end{equation*}
Hence, by virtue of Lemma \ref{lemma-derivative},  if we consider $W(t):=\sqrt{V(t)}$, then one has
\begin{equation*}\label{inequality-V2}
\begin{aligned}
\dot{W}(t)&\leq \sqrt{2}\left( d_{r(T)}(f_1,f_2)(t)+\int_{T_0}^tD_{r(T)}(g_1,g_2)(t,s)ds\right)\\
&+2\left(\frac{\nu(t)}{\rho}+\kappa_{r(T)}\right)W(t)+2\int_{T_0}^t \mu_{r(T)}(s)W(s)ds.
\end{aligned}
\end{equation*}
Therefore, by virtue of the enhanced Gronwall's inequality I (Theorem \ref{TeoI} ), we obtain that
$$
W(t) \leq W(T_0) \exp\left(\int_{T_0}^t \delta(s)ds\right)+\int_{T_0}^t \Delta(s)\exp\left(\int_s^t \delta(\tau)d\tau\right)ds,
$$
which proves the result. 
\end{proof}

\section{Integro-differential projected dynamical systems}

Let $\H$ be a real Hilbert space and consider the following integro-differential equation:
\begin{equation}\label{Integro-diff}
\left\{
\begin{aligned}
\dot{x}(t)&=f(t,x(t))+\int_{T_0}^t g(t,s,x(s))ds &  \textrm{ a.e. } t\in I,\\
x(T_0)&=x_0,
\end{aligned}
\right.
\end{equation}
where $f\colon I\times \H \to \H$ and  $g\colon I\times I\times \H \to \H$ are functions satisfying \ref{Hf} and \ref{Hg}, respectively. According to Theorem \ref{Existence}, the above problem has a unique absolutely continuous solution $x\colon I \to \H$. Moreover, Theorem \ref{Well-posedness} provides a quantified continuous dependence result for \eqref{Integro-diff}.  Hence, the problem \eqref{Integro-diff} is well-posed. 

Given a nonempty and  closed set $C\subset \H$ and $x_0\in C$, the viability problem for \eqref{Integro-diff} related to $C$ consists in providing necessary and sufficient conditions on $C$, $f$ and $g$ in order that the unique solution $x(\cdot)$ of  \eqref{Integro-diff} stays on the set $C$, that is, $x(t)\in C$ for all $t\in I$. For classical differential equations ($g\equiv 0$), the first result in this sense goes back to Nagumo \cite{MR15180} who proved  that a necessary and sufficient condition for viability is that the right-hand side of the differential equation belongs to the tangent cone of the viability set $C$.  Since then, a well-developed viability theory has been developed for differential inclusions, including the case of differential equations. We refer to \cite{MR1389140,MR2815180} for more details. 
 However, viability theory for integro-differential equations seems inadequate due to the memory effects induced by the integral term on \eqref{Integro-diff}. Indeed, roughly speaking, a tangential condition for the viability result would require that the right-hand side of \eqref{Integro-diff} belongs to the tangent cone of $C$, involving the solutions of the dynamical system, which is impractical from a useful point of view. One way to overcome this difficulty is to consider a related dynamical system whose solutions naturally stay within the given set $C$. Although this changes the system under study, it can still provide important information about the phenomenon of interest. 
 
 In what follows, we assume that $C$ is a nonempty, closed and convex set and $x_0\in C$. We follow the developments of the theory of \emph{projected dynamical systems} (see, e.g., \cite{MR2187207,MR2185604}) and consider the following integro-differential projected dynamical system:
 \begin{equation}\label{Projected2}
\left\{
\begin{aligned}
\dot{x}(t)&=\operatorname{proj}_{T_C(x(t))}\left(f(t,x(t))+\int_{T_0}^t g(t,s,x(s))ds\right) &  \textrm{ a.e. } t\in I,\\
x(T_0)&=x_0\in C,
\end{aligned}
\right.
\end{equation}
where $T_C(x(t))$ denotes the tangent cone of the set $C$ at $x(t)$. When $g\equiv 0$ the above system is called  projected dynamical system  and have been used to model several equilibrium problems (see \cite{MR2187207} and the references therein).  Before establishing the well-posedness for the above dynamical system, we provide the equivalence of \eqref{Projected2} with the following integro-differential sweeping process of Volterra type:
 \begin{equation}\label{SP-Proj}
\left\{
\begin{aligned}
\dot{x}(t)&\in -N_C(x(t))+f(t,x(t))+\int_{T_0}^t g(t,s,x(s))ds &  \textrm{ a.e. } t\in I,\\
x(T_0)&=x_0\in C.
\end{aligned}
\right.
\end{equation}
The following result follows in the same way as \cite[Corollary~2]{MR2185604}
\begin{proposition}\label{rela} Let $C$ be a nonempty, closed and convex set. Then, any solution of \eqref{Projected2} is a solution of \eqref{SP-Proj}. Reciprocally, any slow solution of \eqref{SP-Proj} is a solution of \eqref{Projected2}.
\end{proposition}
It is worth emphasizing that due to the presence of the normal cone, any solution of \ref{SP-Proj} satisfies the constraint $x(t)\in C$.  Now, we are in position to prove the existence of solutions for \eqref{Projected2}.
\begin{proposition}
Let $C$ be a nonempty, closed and convex set and assume that \ref{Hf} and \ref{Hg} hold.  Then, for any $x_0\in C$, there exists a unique absolutely continuous solution $x(\cdot)$ of \eqref{Projected2}. Moreover, $x(\cdot)$ is a slow solution of \eqref{SP-M} and the following bounds hold
\begin{equation*}
\begin{aligned}
\Vert x(t)\Vert &\leq r(t):=\Vert x_0\Vert \exp\left( \int_{T_0}^t \gamma(s)ds\right)+\int_{T_0}^t \gamma(s)\exp\left(\int_s^t \gamma(\tau)d\tau \right)ds &\textrm{ for all } t\in I,\\
\Vert \dot{x}(t)\Vert & \leq \theta(t):=\gamma(t)(1+r(t)) &\textrm{ for a.e. } t\in I,
\end{aligned}
\end{equation*}
where $\gamma(t):=2\beta(t)+2\int_{T_0}^t \sigma(t,s)ds$.
\end{proposition}
\begin{proof} It follows directly from Theorem \ref{Existence-2} and Proposition \ref{rela}.
\end{proof}

\section{Concluding remarks}
In this paper, we have studied a variant of Moreau's sweeping process known as the integro-differential sweeping process of Volterra type. After establishing new enhanced versions of Gronwall's inequality, we provided an existence result for integro-differential sweeping process of Volterra type. Our approach is based on a reparametrization technique and a fixed point argument for history-dependent operators. We observe that the main result from Section \ref{sec2} (Theorem \ref{Existence}) is based on an existence result for the sweeping process due to Edmond and Thibault \cite{MR2179241}, where the prox-regularity and a priori bound on the derivative of the trajectory is provided. Hence, by using the same fixed point argument with a different existence result for Moreau's sweeping process (for example, with truncated bounded retraction), more general results for the integro-differential sweeping process of Volterra type can be obtained.   Another key point of this work is the fully continuous dependence result obtained in Section \ref{sec3}. Indeed, we provide estimates of the dependence of the solutions with respect to all the data of the problem. Again, our analysis is based on new enhanced versions of Gronwall's inequality. 

Several open questions remain. On the one hand, a natural continuation of this work is to consider integro-differential state-dependent sweeping process of Volterra type with nonconvex sets, e.g., submooth moving sets (see, e.g., \cite{MR3626639}). Degenerate versions of integro-differential state-dependent sweeping processes could also be explored  (see, e.g.,  \cite{MR4421900}).

\paragraph{Acknowledgements} This work was supported by the Center for Mathematical Modeling (CMM) and ANID-Chile under BASAL funds for Center of Excellence FB210005, Fondecyt Regular 1200283, Fondecyt Regular 1240120 and Fondecyt Exploración 13220097. We also acknowledge ECOS-Anid Project PC23E11 and MathAmsud Project 23-MATH-17.

\bibliographystyle{abbrv}
\bibliography{ref}
\end{document}